\documentclass[12pt,a4]{amsart}

\usepackage{todonotes}
\usepackage{mathabx} 
\usepackage{amssymb}
\usepackage[colorlinks]{hyperref}
\usepackage{amsrefs}

\newtheorem{lemma}{Lemma}
\newtheorem*{theorem}{Theorem}

\newcommand{\tr}{\mathrm{tr}}
\renewcommand{\k}{\mathbb{F}} 
\newcommand{\dd}{\mathrm{d}}
 
\newcommand{\act}
{\smalltriangleright}
\newcommand\cO{\mathcal{O}}
\newcommand\cA{\mathcal{A}}
\newcommand{\casimir}{\mathcal{C}}

\begin{document}
\title[Hochschild dimension of liberated quantum groups]{A remark on
the Hochschild 
dimension of 
liberated quantum groups}
\author[Brzezi\'nski]{Tomasz Brzezi\'nski}
\address{
Department of Mathematics, Swansea University, 
Swansea University Bay Campus,
Fabian Way,
  Swansea SA1 8EN, U.K.\ \newline \indent
Faculty of Mathematics, University of Bia{\l}ystok, K.\ Cio{\l}kowskiego  1M,
15-245 Bia\-{\l}ys\-tok, Poland}
\email{T.Brzezinski@swansea.ac.uk}
\author[Kr\"ahmer]{Ulrich Kr\"ahmer}
\address{Institut f\"{u}r Geometrie, TU Dresden,  
Dresden, Germany}
\email{ulrich.kraehmer@tu-dresden.de}
\author[\'O Buachalla]{R\'eamonn \'O Buachalla}
\address{Mathematical Institute of Charles University, Sokolovsk\'a 83, Prague, Czech Republic}
\email{obuachalla@karlin.mff.cuni.cz}
\author[K. R. Strung]{Karen R. Strung}
\address{Institute of Mathematics, Czech Academy of Sciences,
\v{Z}itn\'a 25, 115 67 Prague, Czech Republic}
\email{strung@math.cas.cz}

\begin{abstract}
Let $A$ be a Hopf algebra
equipped with a projection
onto the coordinate Hopf
algebra $\mathcal{O}(G)$ of a
semisimple algebraic group $G$.
It is shown that if $A$ admits
a suitably non-degenerate
comodule $V$ and the induced
$G$-module structure of $V$ is
non-trivial, then the third
Hochschild homology group of
$A$ is non-trivial.
\end{abstract}
\maketitle


\section{Introduction}

For a field $\k$, let $\cO(G)$
denote the Hopf algebra of
coordinate (polynomial)
functions on an algebraic
group $G$. Let furthermore
$HH_*(A)$ denote the 
Hochschild homology of an
associative (unital) algebra
$A$ over $\k$ with
coefficients in $A$. In this note we
prove the following:

\begin{theorem}
Let $G$ be a semisimple algebraic
group over a field $\k$ of
characteristic 0,  
$\pi \colon A \longrightarrow
\cO(G)$ be a Hopf algebra map,
and $V$ be a right
$A$-comodule with 
a non-degenerate symmetric or
antisymmetric invariant
bilinear form.
If the representation of $G$ 
on $V$ induced by $ \pi $ is
nont-rivial, then 
$
	HH_3(A) \neq 0.
$ 
\end{theorem}

This theorem is best seen in the context of the {\em liberation}
 procedure \cite{BanSpe:lib}
for compact quantum matrix
groups in the sense of
Woronowicz \cite{Wor:com}.
Although this procedure is not formally defined,
 its origins can be traced
back to the work of Wang
\cite{Wan:fre} on free quantum
groups or even earlier to
\cite{DubLau:deg}. At the
algebraic level, the idea is
to construct for a given
representation $V$ of an
algebraic group $G$ and a
non-degenerate bilinear
form on $V$ a universal 
Hopf algebra map 
$\pi\colon \cA(G) \longrightarrow
\cO(G)$ as in the above
theorem, see 
e.g.~\cite[Theorem~1]{BicDub:gro}.
Following this philosophy,
Wang constructed free quantum
orthogonal  and unitary groups
$A_o(N)$, $A_u(N)$ and
interpreted the $C^*$-algebra
completions in terms of a free product of $C^*$-algebras in \cite{Wan:fre}. The former is a universal $C^*$-algebra generated by $N^2$ elements $a_{ij}$ subject to relations
$$
\sum_{k} a_{ik}a_{jk} = \sum_{k} a_{ki}a_{kj} = \delta_{ij}, \qquad a_{ij}^* = a_{ij}.
$$
Collins, H\"artel and Thom
\cite{CHT} studied the Hochschild
homology of $A_o(N)$ showing
that for all $N \geq 2$ the
third Hochschild homology
group with coefficients in $
\mathbb{C} $ is
one-dimensional and 
that $A_o(N)$ is a Calabi--Yau
algebra of
dimension 3 (the 
homology groups with arbitrary
coefficients vanish in degrees
above 3 and 
satisfy Poincar\'e duality
in the sense of Van den Bergh
\cite{Van:rel}). 
Our theorem shows that this
non-triviality of third
Hochschild homology groups 
has a general
representation-theoretic
explanation.

The liberation procedure can
be extended to intermediate
phases leading, for example,
to half-liberated matrix
quantum groups
\cite{BanSpe:lib} or
half-commutative Hopf algebras
\cite{BicDub:hal}; the theorem
can be applied to these
examples, too.

The proof of the theorem uses elementary
noncommutative geometry: by
choosing a basis
$e_1,\ldots,e_N$ in an
$N$-dimensional comodule
over a Hopf algebra
$A$, one obtains an
invertible matrix 
$v \in GL_N(A)$ with 
$ \rho (e_j) = \sum_i e_i
\otimes v_{ij}$ and hence a
class  
$[v] \in K_1(A)$. The Chern--Connes
character assigns to 
$[v]$ 
classes in the odd cyclic homology 
groups $HC_{2d+1}(A)$. The
main point is that assuming
the existence of a symmetric
or antisymmetric
non-degenerate invariant 
pairing on $V$, 
the class in the cyclic homology group $HC_3(A)$ is in
the image of the natural map 
$HH_3(A) \longrightarrow
HC_3(A)$ (Lemma~\ref{natural}). 
Under $ \pi _*$,
these classes in the K-theory
respectively~cyclic and Hochschild
homology of $\cO(G)$ are 
well-known to be non-trivial
(see the final
Section~\ref{masoud}),
hence the theorem follows.

\section{Preliminaries}
In this section we fix
notation and terminology on
Hopf algebras and homological
algebra. 
All the material is standard,
see e.g.~\cite{sweedler}
respectively
\cite{cartaneilenberg} for
more background and details. 
The theory of 
self-dual comodules
is a slightly more
specialised topic, 
hence we include more details
here.

\subsection{The 
comodule $V$}
Let $A$ be a Hopf
algebra with coproduct
$$ 
	\Delta \colon A \rightarrow
	A \otimes A,
	\quad 
	a \mapsto a_{(1)} 
	\otimes a_{(2)},
$$
counit $
\varepsilon \colon A
\rightarrow \k$, and antipode 
$S \colon A \rightarrow A$ 
over a field $\k$, and let 
$V$ be an $N$-dimensional 
right $A$-comodule with
coaction 
$$ 
	\rho \colon V
	\rightarrow V \otimes A,
	\quad
	e \mapsto e_{(0)} 
	\otimes e_{(1)}.
$$

We fix a vector space basis 
$\{e_1,\ldots,e_N\}$ of $V$
and denote by $\{v_{ij}\}$ the
matrix coefficients of $V$
with respect to this basis,
$$
	\rho (e_j) =  
	\sum_i e_i \otimes v_{ij}.
$$
Then we have
\begin{equation}\label{corep}
	\Delta (v_{ij}) =
	\sum_k v_{ik} \otimes 
	v_{kj},\quad
	\varepsilon (v_{ij}) =
\delta _{ij},
\end{equation}
and the matrix $v \in
M_N(A)$ with entries $v_{ij}$
is invertible with inverse
matrix $v^{-1}$ having the 
$ij$-entry $S(v_{ij})$,
$$
	\sum_k S(v_{ik})v_{kj} = 
	\sum_k v_{ik} S(v_{kj}) = 
	\varepsilon (v_{ij}) = 
	\delta _{ij}.   
$$

\subsection{The pairing 
$ \langle -,- \rangle$}
The comodule $V$ is {\em 
self-dual} if there is a
non-degenerate bilinear
form 
$$
	\langle -,- \rangle 
	\colon V \otimes V
\rightarrow \k ,
$$
which is a morphism of
$A$-comodules, where 
$\k$ carries the trivial
coaction 
$$
	\k 
	\rightarrow \k \otimes A \cong
	A,\quad 
	1=1_\k \mapsto 1=1_A,
$$
that is, if 
$$
	\langle d_{(0)} , e_{(0)}
	\rangle 
	d_{(1)}e_{(1)}
	=
	\langle d, e \rangle 1_A
$$
holds for all $d,e \in V$.

In terms of the basis
$\{e_i\}$, the bilinear form 
$\langle -,- \rangle $ is
determined by the matrix 
$E \in M_N(\k)$ with entries $ \langle
e_i, e_j \rangle $ and 
is non-degenerate if and only
if 
$E \in GL_N(\k)$. Analysing
when it is $A$-colinear
yields:

\begin{lemma}\label{selfdual}
The comodule $V$ is self-dual if and only
if there exists $E \in
GL_N(\k)$ with 
$$
	v^{-1} = E^{-1} v^T E, 
$$
where $v \in M_N(A)$ is as in
(\ref{corep}). 
\end{lemma}

\begin{proof}
Assume that 
$ \langle -, - \rangle $ is
any bilinear form on
$V$. 
In terms of the basis 
$\{e_j \otimes e_s\}$ of 
$V \otimes V$, applying
the $A$-coaction on 
$V \otimes V$ and then the map
$ \langle -,- \rangle
\otimes \mathrm{id} _A$ 
gives
$$
	e_j \otimes e_s 
	\mapsto  
	\sum_{ir} e_i \otimes e_r \otimes 
	v_{ij}v_{rs} 
	\mapsto 
	\sum_{ir} E_{ir}  
	v_{ij}v_{rs}.
$$
Applying instead  
$\langle -,- \rangle$ and then
the (trivial) coaction on 
$\k$ gives $E_{js}$, so
$ \langle -,- \rangle $ is
$A$-colinear if and only if
$$	
	E_{js} = 
	\sum_{ir} E_{ir}  
	v_{ij}v_{rs}
$$
holds for all $1 \le j,s \le N$. 

If this holds, then
multiplying by $S(v_{sk})$ from
the right and summing over $s$
yields 
$$
	\sum_s E_{js} S(v_{sk})= 
	\sum_{irs} E_{ir}  
	v_{ij}v_{rs} S(v_{sk}) 
	=
	\sum_i E_{ik}  
	v_{ij}.
$$
If $E$ is invertible,
multiplying from the left by 
$(E^{-1})_{lj}$ and summing
over $j$ finally yields
\begin{align*}
	(v^{-1})_{lk}&=
	S(v_{lk})=
	\sum_{sj} 
	(E^{-1})_{lj}	
	E_{js} S(v_{sk})\\
	&= 
	\sum_{ij} 
	(E^{-1})_{lj}
	E_{ik}  
	v_{ij} 
=
	(E^{-1}v^TE)_{lk}. 
\end{align*}
Conversely, if there is an
$E \in GL_N(\k)$ with this
property, simply 
define 
$ \langle -,- \rangle $ by
setting $ \langle e_i,e_j
\rangle := E_{ij}$ and then
the above shows that this
renders $V$ self-dual.
\end{proof}

\subsection{The Lie algebra
$ \mathfrak{g} _A$}
The dual vector
space $A'=\mathrm{Hom}
_\k(A,\k)$ is an algebra with 
respect to the convolution 
product 
$$
	(fg)(a):= f(a_{(1)})
	g(a_{(2)}),\quad
	f,g \in A',a \in A, 
$$
and the subspace 
$$
	\mathfrak{g} _A := 
	\{f \in A' \mid 
	f(ab) = \varepsilon (a) 
	f(b) + f(a) \varepsilon
	(b),\,\forall a,b \in A\}
$$
of \emph{primitive elements} in $A'$
is a Lie algebra with Lie
bracket given by the commutator 
$[f,g] := fg-gf$, 
for all $f,g \in \mathfrak{g} _A$.

The right $A$-comodule $V$ is
naturally a left $A'$-module 
via 
$$
	f \act e:=e_{(0)}
	f(e_{(1)}),\quad
	f \in A',e \in V. 
$$
As $A$ itself is also a right
$A$-comodule via $ \Delta $, 
$A$ becomes analogously
a left $A'$-module via 
$$
	f \act a:=a_{(1)}
	f(a_{(2)}),\quad
	f \in A',a \in A.
$$
In particular, this defines an
action of the Lie algebra
$ \mathfrak{g} _A$ of
primitive elements 
$f \in A'$ by $\k$-linear
derivations on $A$:
\begin{align}\label{deriv}
	f \act (ab) 
&= a_{(1)} b_{(1)} 
	f( a_{(2)} b_{(2)})
\nonumber\\
&= a_{(1)} b_{(1)} 
	(\varepsilon (a_{(2)}) 
	f(b_{(2)}) + 
	f(a_{(2)}) \varepsilon (
	b_{(2)} ))\\
&=  
	a (f \act b) + 
	(f \act a) b.\nonumber
\end{align}

\subsection{Hochschild
(co)homology}
We denote by  
\begin{align*}
b_n &\colon A^{\otimes n+1}
	\rightarrow A^{\otimes n}\\
\beta_n &\colon \mathrm{Hom}
_\k(A^{\otimes n},A)
\rightarrow \mathrm{Hom}
_\k(A^{\otimes n+1},A)
\end{align*}
the Hochschild (co)boundary
maps of the algebra $A$ and
by 
$$
	HH_n(A) := \mathrm{ker}\,
	b_n/\mathrm{im}\, b_{n+1},
	\quad
	H^n(A,A) := \mathrm{ker}\,
\beta _n/\mathrm{im}\, \beta
_{n-1}
$$ 
the Hochschild
(co)homology of $A$ with
coefficients in $A$. In particular,
an $\k$-linear derivation of
$A$ is the same as a
Hochschild 1-cocycle, so by
(\ref{deriv}), 
the action of 
primitive elements $f \in
\mathfrak{g} _A$ on $A$
defines a linear map 
$
	\mathfrak{g} _A 
	\rightarrow H^1(A,A). 
$

Recall finally that there are
well-defined cup and cap
products
(see 
e.g.~\cite[Section XI.6]{cartaneilenberg})
\begin{align*}
	\smallsmile 
&\colon
	H^i(A,A) \times 
	H^j(A,A) \rightarrow 
	H^{i+j} (A,A),\\
	\smallfrown 
&\colon
	HH_i(A) \times 
	H^j(A,A) \rightarrow 
	HH_{i-j} (A)
\end{align*}
which at the level of
(co)cycles are given by 
$$
	(\varphi \smallsmile
	\psi) 
	(a_1,\ldots,a_i,b_1,
	\ldots,
	b_j) 
= 
	\varphi(a_1,\ldots,
	a_i) \psi(b_1,
	\ldots,b_j)
$$
and
$$
	(a_0 \otimes \cdots 
	\otimes a_i) \smallfrown 
	\varphi = 
	a_0 \varphi (a_1,\ldots,
	a_j) \otimes a_{j+1}  
	\otimes \cdots 
	\otimes a_i,
$$
and that the cup product is
graded commutative, that is,
for all $[ \varphi ] \in
H^i(A,A), [\psi] \in
H^j(A,A)$, 
\begin{equation}\label{supercommu}
	[ \varphi ] \smallsmile 
	[\psi ] = (-1)^{ij}
	[ \psi ] \smallsmile 
	[\varphi ].
\end{equation}

\section{Proof of the theorem}
In this section we prove the
main theorem. We construct
explicitly a suitable
Hochschild 3-cycle on a Hopf
algebra $A$ and  then show
that it is non-trivial by
pairing it with the Lie algebra of primitive elements in the dual Hopf algebra $A'$.

\subsection{The Hochschild 
3-cycle $ c_V$}
The starting point of the proof of the main result of this 
paper is the following remark
which we expect to be well known
to experts:

\begin{lemma}\label{natural}
Assume $(V, \langle -,- \rangle )$
is a self-dual comodule over $A$. 
If 
$ \langle -,- \rangle $ is 
symmetric or
antisymmetric, then
$$
	c_V:=
	\sum_{ijkl}
	(v^{-1})_{ji} \otimes 
	v_{ik} \otimes 
	(v^{-1})_{kl} \otimes 
	v_{lj} 
	+
	\sum_{ij} 
	1 \otimes v_{ij} 
	\otimes 1 \otimes
(v^{-1})_{ji} 
	\in 
	A^{\otimes 4}
$$
is a Hochschild 3-cycle, 
i.e.,  $b_3 c_V = 0$. 
If $V$ is
simple, then the converse
implication holds as
well. 
\end{lemma}
\begin{proof}
It is straightforward to 
verify that 
$$
	b_3 c_V = 
	\sum_{ij} 
	1 \otimes 
	\bigl((v^{-1})_{ij}  
	\otimes v_{ji}  - 
	v_{ij}  
	\otimes
(v^{-1})_{ji}\bigr), 
$$
and Lemma~\ref{selfdual} yields 
$$
	b_3 c_V = \sum_{ijsr} 
	1 \otimes 
	v_{ij}
	\otimes 
	\bigl(
	E_{ir}  
	v_{rs}  
	E^{-1}_{sj} 
	- 
	E^T_{ir}
	v_{rs} 
	(E^{-1})^T_{sj} 
	\bigr) ,
$$
which vanishes if 
$E^T=\pm E$.  
If $V$ is simple, then 
the $v_{ij}$ are linearly
independent (by the Jacobson
density theorem) and the above
computation shows first that 
$$	
	EvE^{-1}
	=
	E^Tv
	(E^{-1})^T 
	\Leftrightarrow	
	E^{-1}E^T v = 
	v E^{-1} E^T. 
$$
Again by the Jacobson density
theorem and the fact that the
only matrices commuting with
all others are scalar
multiples of the identity
matrix, this implies that $E^{-1}E^T$
is a constant, so
$E^T=\lambda E$ for some 
$ \lambda \in \k$ which is
necessarily $\pm 1$. 
\end{proof}

\subsection{The cap product
$ 
c_V \smallfrown \varphi$}
Let us take any 
$f_1,f_2,f_3 \in  
\mathfrak{g} _A$, 
i.e.~primitive elements of $A'$, and
let $\varphi$ 
be the cup product of the
associated derivations of $A$,
$$
	\varphi \colon 
	A^{\otimes 3} \rightarrow A,
	\quad
	a_1 \otimes a_2 \otimes a_3
	\mapsto 
	(f_1 \act a_1) 
	(f_2 \act a_2)
	(f_3 \act a_3).
$$
We now show that the cap product
between $c_V$ and 
$\varphi$ is a scalar multiple
of the identity $1_A$:

\begin{lemma}\label{cappro}
Let $F_i \colon V \rightarrow 
V,e \mapsto f_i\act e=e_{(0)}
f_i(e_{(1)})$ be the linear
map defined by the action 
of $f_i$, $i=1,2,3$. Then,
$$
	c_V 
	\smallfrown \varphi
	= -\tr(F_1F_2F_3). 
$$
\end{lemma}

\begin{proof}
If $ \partial \colon A
\rightarrow A$ is any
derivation and $v \in
GL_N(A)$, then the Leibniz
rule implies 
$$
	\partial(v^{-1}_{rk}) =
	-\sum_{ij} v^{-1}_{ri} 
	\partial (v_{ij})
v^{-1}_{jk}
$$
and of course $ \partial
(1)=0$. 
Thus 
\begin{align*}
	c_V
	\smallfrown \varphi
&=	\sum_{ijkl} 
	v^{-1}_{ji}
	(f_1
	\act v_{ik}) 
	(f _2 
	\act v^{-1}_{kl}) 
	(f _3 \act v_{lj}) \\
&=	-\sum_{ijklmn} 
	v^{-1}_{ji}
	(f_1 \act
	v_{ik})v^{-1}_{km} 
	(f _2 \act
	v_{mn}) 
	v^{-1} _{nl}
	(f _3 \act v_{lj}) \\
&=	-\sum_{ijklmnpqr} 
	v^{-1}_{ji}
	v_{ip} F_{1,pk}
	v^{-1}_{km} 
	v_{mq} F_{2,qn}
	v^{-1} _{nl}
	v_{lr} F_{3,rj} \\
&=	-\sum_{jkn} 
	F_{1,jk}
	F_{2,kn}
	F_{3,nj}.\qedhere
\end{align*}
\end{proof}

\subsection{Evaluation in 
$ \varepsilon $}
The following is true for any
algebra that admits a
1-dimensional representation:

\begin{lemma}\label{trivia}
The 0-cycle 
$1 \in A$ has a non-trivial
class in $HH_0(A) = A/[A,A]$.
\end{lemma}
\begin{proof}
The counit $ \varepsilon $ 
inevitably vanishes on
all commutators but maps $1_A$
to $1_\k$. 
\end{proof}

\subsection{The Casimir
operator}
In view of Lemma~\ref{trivia},
Lemma~\ref{cappro} implies
$[c_V] \neq 0$ as long as
there are $f_1,f_2,f_3 \in
\mathfrak{g} _A$ with 
$\tr(F_1F_2F_3) \neq 0$. 

This is in particular the case
when $A$ admits a Hopf algebra
map to the coordinate Hopf
algebra of a semisimple
algebraic group $G$ which acts
non-trivially on $V$:
using the graded commutativity 
(\ref{supercommu}) of 
$\smallsmile$ we observe that 
$$
	\tr(F_1 [F_2,F_3]) = 
	\tr(F_1F_2F_3) -
	\tr(F_1F_3F_2) = 2
	\tr(F_1F_2F_3).
$$
Now recall that if $
\mathfrak{g} $ is the Lie
algebra of $G$, then as
$G$ and hence $ \mathfrak{g} $
are semisimple, 
$[\mathfrak{g} ,\mathfrak{g}
]= \mathfrak{g} $ and,
therefore, the (quadratic) 
Casimir operator $ \casimir $
of $ \mathfrak{g}  $ can be
expressed as a finite sum
$$
\casimir = \sum_{m=1}^M 
f_{m1}[f_{m2},f_{m3}], \qquad
f_{mi} \in \mathfrak{g} .
$$
Under the map
$ \pi^* \colon \mathfrak{g}
\rightarrow \mathfrak{g} _A$
dual to $ \pi $ these $f_{mi}$ yield 
primitive elements in $
\mathfrak{g} _A$ and hence
classes $ [\varphi] \in
H^3(A,A) $ 
which add up to a class
whose pairing with $ [c_V]$ is 
$-\frac{1}{2} \tr(\casimir)$. If $G$ acts
non-trivially on $V$, this is 
non-zero, so $[c_V] \neq 0$.

\subsection{The class $ \pi
_*([c_V])$}\label{masoud}
Following a 
suggestion by 
M.~Khalkhali, we end with a
brief historical account on
the role that
$ \pi _*([c_V]) \in HH_3(
\cO(G))$ 
plays in the 
cohomology of algebraic groups and Lie
algebras. For more details, we
refer to Samelson's survey
\cite{samelson}.
 
As we work
over a field of characteristic
0, a semisimple algebraic
group $G$ is a smooth and
connected affine
variety, and the
Hochschild-Kostant-Rosenberg
isomorphism \cite{hkr}
$$ 
	HH_n(\cO(G)) \cong 
	\Omega^n (G)
$$
identifies $ \pi_*([c_V])$ 
with the (K\"ahler) differential 
3-form 
\begin{equation}\label{omegav}
	\omega _V := 
	\sum_{ijkl} 
	g^{-1}_{ji} 
	\dd g_{ik} 
	\dd g^{-1}_{kl} 
	\dd g_{lj} \in \Omega ^3(G) 
\end{equation}
on $G$; here $g_{ij } := 
\pi (v_{ij}) \in \cO(G)$ are the matrix
coefficients of the
representation $ G \rightarrow 
GL(V) \cong GL_N(\mathbb{F})$ 
defined by $ \pi $ and
$ \rho $. 
Note that 
$ 
	\dd g_{ij}^{-1} = \sum_{rs} 
	g_{ir}^{-1} (\dd g_{rs})
	g^{-1}_{sj}$ and that this
implies 
$ \dd \omega _V= 0 $.
When $ \pi =
\mathrm{id}_{\cO(G)} $, 
$V = \mathfrak{g} $, and $
\rho $ is the adjoint
representation of $G$, then 
$ \langle -,- \rangle $ can be
taken to be the Killing form.
Thus every semisimple
algebraic group
$G$ over a field of
characteristic 0 
comes equipped with a
canonical de Rham cohomology 
class $[ \omega _
\mathfrak{g} ] \in HdR^3(G)$.

One of the main results of 
Chevalley and Eilenberg's
seminal paper \cite{ce} was
that this cohomology class is
non-trivial. 
Hopf had shown in \cite{hopf}
that the de Rham 
cohomology of a compact
and connected Lie group is
that of a product of
odd-dimensional spheres
$S^{m_i}$, and for the
classical matrix groups the
$m_i$ had been already known
earlier. 
Chevalley and
Eilenberg then fully
implemented an idea that goes
back to Cartan: 
the cotangent bundle of
an algebraic group $G$ admits a
natural trivialisation, 
$
	T^* G \cong G \times
	\mathfrak{g} '. 
$
From a
Hopf-algebraist's perspective,
this stems from the fact
that $\Omega^1 (G)$ and hence
also $\Omega(G) = \Lambda
_{\cO(G)} \Omega^1 (G)$ is  
a Hopf module over 
$\cO(G)$, hence by the
fundamental theorem of Hopf
modules \cite[Theorem~4.1.1]{sweedler},
$\Omega (G) $
is a free $\cO(G)$-module 
with a basis given by
the elements that are invariant
under the $\cO(G)$-coaction.
Geometrically, this coaction 
is the
$G$-action on differential
forms given by right
translation, hence the basis
elements are the right-invariant
differential forms. By
evaluation in the unit element 
$e \in G$, these become
identified with elements of 
the exterior algebra 
$ \Lambda _\mathbb{F}
\mathfrak{g} '$ of the dual of
the Lie algebra $ \mathfrak{g}
$. The de Rham
differential is
$G$-equivariant, hence
restricts to the
right-invariant differential
forms, and under the
isomorphism becomes the
Chevalley-Eilenberg
differential on $ \Lambda
_\mathbb{F} \mathfrak{g} '$
that computes the Lie algebra
cohomology $ H( \mathfrak{g} ,
\mathbb{F})$ \cite[Theorem
9.1]{ce}.
Furthermore, the differential
forms which are not just right- but
also left-invariant become
identified with the 
$ \mathrm{ad}$-invariant
cochains in the
Chevalley-Eilenberg complex,
and on these the coboundary
map is trivial
\cite[(19.2)]{ce}.  
As $G$ is reductive, this
subcomplex of biinvariant
differential forms is actually
quasiisomorphic to the de Rham
complex, so the de Rham
cohomology of $G$ can be
identified with the 
algebra of
$\mathrm{ad}$-invariant
Chevalley-Eilenberg cochains. 
If we consider compact Lie groups 
over $\mathbb{F} =
\mathbb{R}$, then the 
statements carry over to 
smooth funtions and
differential forms and the de
Rham complex is quasiisomorphic to the
subcomplex of biinvariant differential
forms, which are
automatically closed
\cite[(12.3)]{ce}.  
 
Our form $ \omega _V$ (and in
fact the Hochschild cycle $c_V$)  
is manifestly biinvariant: replacing 
the function $g_{ij} \in
\cO(G)$ by $ \sum_r g_{ir} 
t_{rj} $ or $\sum_s t_{is}
g_{sj}$ for a constant matrix 
$T \in GL_N( \mathbb{F})$
with entries $t_{ij}$
yields the same form 
$ \omega _V$. In our proof
above, we have applied 
$ \pi ^* \colon \mathfrak{g}
\rightarrow \mathfrak{g}_A$ 
to Lie algebra elements 
$f_j \in \mathfrak{g} $
and
then computed the pairing 
of $ [\varphi] $ with $[c_V]$.
By very definition, this
amounts to pairing the
corresponding Hochschild
3-cocycle on $\cO(G)$ with 
$ \pi _*( [c_V])$, and under
the
Hochschild-Kostant-Roenberg
isomorphism, this Hochschild
3-cocycle is the
right-invariant multivector field
$f_1 \wedge f_2 \wedge f_3$ on
$G$. 
That is, our
computation can indeed 
be reinterpreted in
terms of Lie algebra
cohomology as the evaluation
of the Chevalley-Eilenberg
cocycle 
\begin{equation}\label{cecocycle}
	\chi _V \colon 
	\mathfrak{g} \otimes_\mathbb{F} 
	\mathfrak{g} \otimes_\mathbb{F} 
	\mathfrak{g} \rightarrow 
	\mathbb{F},\quad
	f_1 \otimes 
	f_2 \otimes 
	f_3 \mapsto 
	\tr(F_1F_2F_3),
\end{equation}
where 
$$
	F_j := (\dd_e\rho) (f_j)
\in \mathfrak{gl}(V) \cong
	M_N( \mathbb{F})
$$
are the values of the $f_j$
under the representation 
$ \mathfrak{g} \rightarrow 
\mathfrak{gl} (V) \cong 
M_N( \mathbb{F})$
corresponding to the
representation $ G \rightarrow 
GL(V)$. 
 
Note that our assumptions on
$V$ enter the fact that 
$ \chi _V$ is a 3-cocycle: 

\begin{lemma}
If $E \in GL_N(\mathbb{F})$ 
is an invertible matrix 
and $ \mathfrak{g} $ is a Lie
subalgebra of 
$\{ 
F \in M_N(\mathbb{F}) \mid 
F^T = - E F E^{-1}\}$, then 
$$
	\chi \colon 
	\mathfrak{g} \otimes
_\mathbb{F} 
	\mathfrak{g} \otimes
_\mathbb{F} 
	\mathfrak{g} \to \mathbb{F},\quad 
	\chi (F_1,F_2,F_3):=
	\tr(F_1F_2F_3)
$$
is an 
$\mathrm{ad}$-invariant 
cocycle in 
$ C^3( \mathfrak{g} ,
\mathbb{F}) = \Lambda _
\mathbb{F} \mathfrak{g} '$. 
\end{lemma}
\begin{proof}
That $ \chi $ is an
alternating 3-form is seen as
follows: 
\begin{align*}
	\tr ( F_1 F_2 F_3  )  
&= 
	\tr ( F_3^T F_2^T  F_1^T)
	= 
	-\tr (EF_3F_2F_1E^{-1}) 
\\
&=  
	-\tr ( F_3 F_2  F_1) = 
	-\tr (F_2 F_1 F_3).
\end{align*}
That $ \chi $ satisfies the cocycle
condiion 
\begin{align*}
	0 &=
	-\chi ([F_1,F_2],F_3,F_4) 
	+\chi ([F_1,F_3],F_2,F_4) 
	-\chi ([F_1,F_4],F_2,F_3)
\\
& 
	-\chi ([F_2,F_3],F_1,F_4) 
	+\chi ([F_2,F_4],F_1,F_3) 
	-\chi ([F_3,F_4],F_1,F_2) 
\end{align*}
is shown similiarly. The 
$\mathrm{ad}$-invariance is
immedate.
\end{proof}

In this way, the condition on 
$V$ to carry a 
non-degenerate symmtric or
antismmetric invariant
bilinear form can also be
motivated from the point of
view of Lie algebra
cohomology.



\section*{Acknowledgements}
We would also like to
thank M.~Khalkhali, who in his role
as editor has suggested to
add the final Section~\ref{masoud}.
The results reported here were
obtained during the authors' stay
at ICMS Edinburgh within the
Research in Groups Programme in June 2022.
We would like to thank the
International Centre for Mathematical Sciences for financial and
administrative support. The
research of T.\ Brzezi\'nski
is supported in part by the
National Science Centre,
Poland, grant no.\
2019/35/B/ST1/01115.
U.~Kr\"ahmer is supported
by the DFG grant
``Cocommutative comonoids''
(KR 5036/2-1). 
R.~\'O Buachalla is supported
by the Charles University
PRIMUS grant ``Spectral
Noncommutative Geometry of
Quantum Flag Manifolds'' 
PRIMUS/21/SCI/026.
K.R.\ Strung is supported by
GA\v CR project 20-17488Y and
RVO: 6798584.

\end{document}